\documentclass[draft, 10pt]{amsart}

\usepackage[applemac]{inputenc}

\usepackage{amsfonts}
\usepackage{latexsym, amssymb}
\newcommand{\field}[1]{\mathbb{#1}}
\newcommand{\R}{\field{R}}

\def\S{\ifmmode{\rm I\mkern-8.1mu
S\mkern1mu}\else{\rm I\kern-.18em R\hskip1pt\
}\fi\relax}

\def\D\theta ij#1{\dis \frac{\partial #1}{\partial \theta_i^j}}

\def\sqr{{\hskip1pt\vcenter{\vbox{\hrule height.4pt
\hbox{\vrule width.4pt height4pt\kern4pt
\vrule width.4pt}
\hrule height.4pt}}}}

\def\cal{\mathcal}

\def\trans{\ifmmode{\rm \frown\mkern-16.8mu \vert
\mkern8mu}\else{$\frown\mkern-16.8mu\vert\mkern8mu$}\fi\relax}
\def\ap{\rightarrow}

\def\dis{\displaystyle}

\def\D{\Delta}
\def\dis{\displaystyle}
\def\a{\alpha}

\def\g{\gamma}
\def\G{\Gamma}
 \def\d{\delta}

\def\l{\lambda}
 \def\n{\nu}
\def\r{\rho}

\newtheorem{The}{Theorem}[section] \newtheorem{Pro}{Proposition}[section]
\newtheorem{Def}{Definition}[section] 
\newtheorem{Lem}{Lemma}[section]
\newtheorem{Rem}{Remark}[section]
\newtheorem{Not}{Notations}[section]

\newtheorem{Com}{Comment}[section]

\title{ Articulated  arm and special multi-flags (corrected version)  }

\begin{document}
\maketitle
 \begin{center}\author{F\textsc{ernand} PELLETIER}\footnote{Universit\'e de Savoie, Laboratoire de Math\'ematiques (LAMA)
Campus Scientifique,
73376 Le Bourget-du-Lac Cedex, France.
Mail: pelletier@univ-savoie.fr.}\& M\textsc{ayada} SLAYMAN \footnote{Lebanese University, Mathematics Department, Faculty of Sciences, Lebanon; Mail: mslayman@ul.edu.lb \hfill$\;\;{}$}
 \end{center}

\begin{abstract}
In this paper we give a kinematical illustration  of some distributions called special multi-flags distributions. Precisely we define the kinematic model  in angular coordinates of an articulated arm  constituted of a series of $(n+1)$  segments in  $\R^{k+1}$ and construct the special  multi-flag distribution associated to this model. \end{abstract}
\section{\bf Introduction }\label{intro}

The  kinematic evolution of a car towing $n$ trailers can be described by a  Goursat distribution  on the   configuration space  $M=\R^2\times
 (\mathbb{S}^1)^{n+1}$. 
 A Goursat distribution is a  rank-$(l-s)$ distribution  on a manifold M of dimension $l\geq2+s$,  such that each element of its flag of  Lie squares,
 $$D=D^s\subset D^{s-1}=[D^s,D^s]\subset\cdots\subset D^{j-1}=[D^{j},D^j]\subset\cdots\subset D^0=TM$$  is of codimension $1$ in the following one. 
	
Since 2000, Goursat distributions  were generalized in many works (\cite{KumpRub}, \cite{MonZhi3}
\cite{Mor5}, \cite{Mor6}, \cite{Mor8}, \cite{PasResp5}).	
 Special $k$-flags ($k\geq 2$),  which are considered to be extensions of Goursat flags,  were defined in \cite{KumpRub},\cite{Mor6}, and \cite{PasResp5} in several equivalent ways. All these approaches can be reduced to one transparent definition (see \cite{Ad}, \cite{ShYa}). A special $k$-flag of length $s$ on a manifold  $M$ of dimension $(s+1)k+1$ is a sequence of distributions
\begin{center} $D^s\subset D^{s-1}=[D^s,D^s]\subset\cdots\subset
 D^{j-1}=[D^{j},D^j]\subset\cdots\subset D^0=TM$ \end{center}
such that the  respective dimensions of   $D^s,\; D^{s-1},\;\cdots,\;D^0$
are $k+1,\;2k+1,\;\cdots,$ $(s+1)k+1$, for  $j=1,\cdots,s-1$, the Cauchy-characteristic subdistribution $L(D^j)$ of $D^{j}$ is included in $D^{j+1}$ of constant corank one,
$L(D^s)=0$, and there exists a completely integrable subdistribution $F \subset D^1$ of corank one in $D^1$. The integer $k$ is called width.\\

The purpose of this work is to show that the problem of modeling car towing $n$ trailers can be  generalized to  the problem of modeling kinematic problem for an “articulated arm" constituted of ($n+1$) segments in $\R^{k+1}$, such that to this model is naturally associated a special $k$-flag.

In the following, an “articulated arm" of length ($n+1$) is a series of $(n+1)$  segments $[M_{i};M_{i+1}]$, $i=0,\cdots,n-1$,  in $\R^{k+1}$,   keeping a constant length $l_i$, and the articulation occurs at  points $M_i$,  for $i=1,\cdots,n$.\\

It is proposed to study the kinematic  evolution of such a mechanical system under the constraint   that the velocity of each point $M_i$  is collinear with the segment $[M_i,M_{i+1}]$, for  $i=0,\cdots,n$. Such a system is also studied in \cite{Li} and is called  a "n-bar system". In this paper we define precisely the kinematic evolution  of this mechanical system in term of hyperspherical coordinates and we construct the special multi-flag  naturally associated to this model.\\

For $k=1$,  an articulated arm of length $n$ is a modeling problem of a car with $n$  trailers. (see \cite{Jea}). When the number of trailers is large,  this problem can be considered as an
approximation of the “nonholonomic  snake" in the plane (see \cite {Ro}  for instance).   For $k>1$ we can also consider  a  “snake"  in $\R^{k+1}$ (see \cite{Rod} for a complete description). Again, an articulated arm  of length $n$, for $n$ large,  can be considered as a discretization  of a nonholonomic snake  in $\R^{k+1}$.
For instance, in $\R^3$, some problems of  "towed cable" can  model in such a way (\cite{Mau}, \cite{Ro})
\\

In section $2$, we recall the classic context  of the car with $n$ trailers and its interpretation in terms of Goursat distribution. The articulated arm  system is developed in section $3$  and also we  show how to associate a special multi-flags to such a system in cartesian coordinates. In section $4$, we gives a version of the kinematic evolution of an articulated arm in terms of angular coordinates and we get  a generalization of the classical model  of the car with $n$ trailers. The last two sections are devoted to the proofs of the results.

\section{ \bf The car with n trailers }\label{contres}
In this  section we will recall  some fundamental results about  the system of the car with $n$ trailers and its relation with the Goursat distribution. All these results are now classical and can be found in a large number of papers as \cite{FLIE}, \cite{Jea}, \cite{MonZhi1}, \cite{PasResp1}, \cite{Sor1} and many others . \\
\subsection{\bf Notations and equations.}\label{dyna}${}$\\
A car with $n$ trailers is a configuration of  $(n+1)$ trailers in the
 $\R^2$-plane, denoted by $M_0,M_1,\cdots,M_n$,  and  keeping a constant length  between each two trailers. It is proposed to study the  kinematic evolution  of the trailer  $M_n$ with the constraint that the motion is controlled by the evolution of  $M_0$ which symbolize the car. We will use  the same  representation  as M. Fliess \cite{FLIE} and O.J. Sordalen \cite{Sor1} where the car is represented by two driving wheels connected by an axle. It's a kinematic problem with non integrable constraints (i.e. a nonholonomic system) due to the rolling without sliding of the wheels.
 The configuration space  of the system is characterized  by the two dimensional coordinates of $M_n$ and $(n+1)$ angles, whereas there are only two inputs, namely one tangential velocity and one angular velocity which represent the action on the steering wheel and on the accelerator of the car.\\
 Consider the system of the car with $n$ trailers and suppose that the distances $R_r$ between the  different trailers  are all equal to $1$. We choose  as a reference point of a body  $M_{n-r}$  the midpoint $m_r$ between the wheels; its  coordinates are denoted by $x_r$ and $y_r$ in a given cartesian frame of the plane; $\theta_r$ is the angle between the main axis of  $M_{n-r}$ and the $x$-axis of the frame.
 So, the set of all positions of the car with with $n$ trailers is  included in a $3(n+1)$-dimensional space.  This system is submitted to $2n$ holonomic links  which give, in the previous  space,  the $2n$ following equations:
\begin{equation}\label{etat}
\begin{array}{ccc}
x_{r}-x_{r-1}&=&\cos \theta_{r-1}\\
y_{r}-y_{r-1}&=&\sin \theta_{r-1}
\end{array}
\end{equation}

 The configuration  space of this problem is a submanifold  of  dimension $(n+3)$ which  is  parameterized  by
$q=(x_0, y_0, \theta_0, ..., \theta_n)$ where : \\
${}\;\;\;\bullet$ $(x_0, y_0)$ are the coordinates of the last trailer $M_n$.\\
${}\;\;\;\bullet$ $\theta_n$ is the orientation of the car (the trailer $M_0$) with respect to the  $x$-axis.\\
${}\;\;\;\bullet$ $\theta_r$, $0 \leq r\leq n-1$, is the orientation of the trailer $(n-r)$ with respect to the $x$-axis.\\
The configuration space can thus be identified to $\mathbb{R}^2 \times (\mathbb{S}^1)^{n+1}$. \\The velocity parameters are  $\dot{x_0}$, $\dot{y_0}$, $\dot{\theta}_{0}$, $\cdots$, $\dot{\theta}_{n}$. There are only two inputs, namely  the  “angular velocity" $w_n$ and  the “tangential velocity" $v_n$
of  the midpoint of the guiding wheels associated to the action of the car (see \cite{Jea}). \\
Assume that the contacts between the wheels and the ground are pure rolling, it is then submitted to the classical nonholonomic links:
\begin{equation}\label{cont}
\dot{x_{r}}\sin\theta_{r}-\dot{y_{r}}\cos\theta_{r} = 0
\end{equation}
 There are $(n+1)$ kinematic equality  constraints, one for each trailer.
In order to establish these constraints, we can represent the points $m_r$, $r=0, \cdots, n$, in the  complex plane, i.e, $m_{r}=x_{r}+iy_{r}$. The geometric constraint between two  consecutive trailers is written as:
\begin{center}
$m_{r}=m_{r-1}+e^{i \theta_{r-1}}$ \;\;\;\;\;for $r \ne 0$
\end{center}
By induction, we have the following equation:
\begin{equation}\label{geom}
m_{r}=m_{0}+\dis\sum_{l=0}^{r-1}e^{i\theta_{l}}
\end{equation}
The kinematic constraint of $M_{n-r}$ is :
$$\dot{m_{r}}=\lambda_{r}e^{i\theta_{r}}$$
which is equivalent to :
$${\cal I}(e^{(-i\theta_{r})}\dot{m_r})=0$$
 where ${\cal I}(z)$ denotes the  imaginary part of  $z$. Combining this  characterization with  the derivative of  (\ref{geom}) and using the  linearity of  ${\cal I}$, we obtain the kinematic constraints:
\begin{equation}
-\dot{x_0}\sin\theta_{r}+\dot{y_0}\cos\theta_{r}+
\Sigma_{j=0}^{r-1}\dot{\theta_{j}}\cos(\theta_{j}-\theta_{r})=0  \;\;\;\;\;\;\;r=0,\cdots, n
\end{equation}
Combining $\dot{m_{r}}=\lambda_{r}e^{i\theta_{r}}$ with the derivative of  $$|m_{r+1}-m_{r}|^2=1$$
we obtain
$$\lambda_{r}=\lambda_{r+1}(\cos\theta_{r+1}-\cos\theta_{r})$$
and by induction:
$$\lambda_{r}=\lambda_{n}\cos(\theta_{n}-\theta_{n-1})\cdots\cos(\theta_{r+1}-\theta_{r})$$
so
$$\dot{m_{r}}=\lambda_{n}(\dis\prod_{j=r+1}^{n}\cos(\theta_{j}-\theta_{j-1}))e^{i\theta_{r}}$$
 where $\lambda_{n}= v_{n}$ is the  tangential velocity of the car  $M_{0}$.
\\

{ \it The evolution of the system of car with $n$ trailers  can be given by the following controlled system with two controls $v_n$ (“tangential velocity") and $w_n$ (“normal velocity")} of $M_0$:
\begin{eqnarray}
\label{systcont}
\begin{cases}
\dot{x_0}= v_0\cos\theta_0 \cr
\dot{y_{0}}=v_{0}\sin\theta_0\cr
\dot{\theta}_{0}=v_{1}\sin (\theta_{1}-\theta_{0})\cr
\qquad \cdots\cr
\dot{\theta}_{r}=v_{r+1}\sin(\theta_{r+1}-\theta_{r})\cr
\qquad \cdots \cr
\dot{\theta}_{n-1}=v_{n}\sin(\theta_{n}-\theta_{n-1})\cr
\dot{\theta}_{n}=w_{n}\cr
\end{cases}
\end{eqnarray}

The  “tangential  velocity"  $v_r$ of the body  $M_{n-r}$ is given by :
\begin{center}
$v_r=\dis\prod^{n}_{j=r+1} cos(\theta_j-\theta_{j-1})v_n$
\end{center}
\subsection{\bf Goursat flag}${}$\\
Given a smooth distribution $D$ on a manifold $M$ we will use the standard notation $[D, D]$  to denote  the smooth distribution generated by the vector fields tangent to $D$ and the Lie brackets $[X,Y]$, of any pair $(X, Y)$  of vector fields tangent to $D$.

\begin{Def}
A Goursat flag of length $s$ on a manifold $M$ of dimension $l\geq s+2$ is a sequence of distributions on $M$
$$D^{s} \subset D^{s-1}\subset \cdots \subset D^{3}\subset D^{2}\subset
D^{1}\subset D^{0}=TM \;\;\;\; s\geq2\;\;\;\;\;\;\;\;\;\;\; {\it (F)}$$
satisfying the following Goursat conditions \\\\
$\begin{array}{ccc}
1)\; corang\; D^{i}= i & i=1, 2, \cdots, s\\
{}\\
\;\;2)\;D^{i-1}=[D^i, D^i]& i=1, 2, \cdots, s
\end{array}$ \;\;\;\; {\it (G)}
\end{Def}

Each $D^i(p)$ is a subspace of $T_pM$ of codimension $i$, for any point $p\in M$. It follows that $D^{i+1}(p)$ is a hyperplane in $D^i(p)$, for any $i=0, 1, \cdots, s-1$ and $p\in M$.\\
\begin{Def} We call  any distribution $D^i$ of corank $i\geq 2$ in a Goursat flag {\it(F)} a Goursat distribution.
 \end{Def}

 To each flag ${\it (F)}$ of Goursat distributions we associate a flag of “Cauchy-characteristic" subdistributions
 \begin{center} $L(D^s) \subset L( D^{s-1})\subset \cdots \subset L(D^{3})\subset L(D^{2})\subset
L(D^{1})$\;\;\;\;\;\;\;\;\;\; {\it (L)}\end{center}
where $L(D)$ is the  subdistribution  of  $D$ generated by  the set of vector fields  $X$  tangent to $ D$  such that  $[X,Y]$ is tangent to $D$ for all $Y$ tangent to  $D$. $L(D)$ is called the Cauchy-characteristic distribution of $D$.
\begin{Lem}(Sandwich lemma)\cite{MonZhi1}: Let $D$ be any Goursat distribution of corank $s\geq2$ on a manifold $M$, and $p$ be any point of $M$. Then \begin{center} $L(D)(p)\subset L([D,D])(p)\subset D(p)$,\end{center} with \;\; dim $L(D)(p)$= dim $D(p)-2$, \;\; dim $L([D,D])(p)$ = dim $D(p)-1$.
\end{Lem}

It follows a relation between the Goursat flag and its flag of Cauchy-characteristic subdistributions:

$$\begin{array}{ccccccccccccccccc}
& &D^s& \subset &D^{s-1}&\subset &
\cdots&\subset&D^3&\subset&D^2&\subset&D^1&\subset&D^0\\
& & \cup& & \cup & & & & \cup& &\cup& \\
 L(D^s)& \subset& L( D^{s-1})&\subset& L(D^{s-2})&\subset &\cdots&\subset& L(D^{2})&\subset&
L(D^{1})\\
\end{array}$$

Each inclusion here is a codimension one inclusion of  subbundles of the tangent bundle. $L(D^i)$ is an involutive regular distribution on  $M$ of codimension $i+2$.

\subsection{\bf Goursat flag associated to the car with n trailers}\label{goursatcar}${}$\\
Let $f^n_r=\dis\prod^{n}_{j=r+1} cos(\theta_j-\theta_{j-1})$,\\
and $v_r=f_r^nv_n$ for $r=1, \cdots, n-1$ \\
The motion of the system associated to the car  is then characterized by the equation:
$$\dot{q}=w_nX^{1}_{n}(q)+v_nX^{2}_{n}(q)$$
It is a controlled system with controls $v_{n}$ and $w_n$,  ($v_{n}$ is the tangential velocity and $w_n$ is the angular velocity  as we have already seen at the beginning of the section).  Each trajectory of the kinematic evolution of the car towing $n$ trailers is an integral curve of the 2-distribution, on $\R^2 \times (\mathbb{S}^1)^{n+1}$, generated by:
$$\left\{
\begin{array}{ccl}
X^{1}_{n}&=&\dis\frac{\partial}{\partial \theta_n}\\
 & & \\
X^{2}_{n}&=&\cos{\theta_{0}} f_{0}^{n}\dis\frac{\partial}{\partial x}+ \sin{\theta_{0}} f_{0}^{n}\dis\frac{\partial}{\partial y}
+\sin(\theta_1-\theta_0)f_1^n\dis\frac{\partial}{\partial\theta_0}
+\cdots+\sin(\theta_{n}-\theta_{n-1})\dis\frac{\partial}{\partial\theta_{n-1}}
\end{array}
\right.$$

{\it  The distribution generated by  $\{X_{1}^{n}, X_{2}^{n}\}$, naturally associated to the system of the car with $n$ trailers, is a Goursat distribution.}

\section{ \bf Articulated arm}
The  purpose of this section is to construct  a distribution $\D$, of dimension $k+1$, naturally associated to an $(n+1)$ articulated arm, and  which generates a special $k$-flag of length $(n+1)$ on the configuration space ${\cal C}\equiv \R^{k+1}\times(\mathbb{S}^k)^{n+1}$. Moreover, the kinematic evolution of this arm is  an integral curve of $\D_n$. We begin by recalling the context of special multi-flag in the formalism of \cite{Ad}, \cite{Mor6}.

\subsection{\bf Special multi-flags}\label{multiflag}${}$\\

A special $k$-flag of length $s$ is a sequence
$$D=D^{s}\subset\hfill D^{s-1}\hfill\subset\cdots\subset D^{j}\subset\cdots\subset D^{1} \subset  D^{0}= TM$$
of distributions on a manifold $M$ of dimension $ (s+1)k+1$ which satisfies the following conditions:\\
(i) $D^{j-1}=[D^j,D^j]$\\
(ii) $ D^s,\; D^{s-1},\;\cdots,\; D^j,\;\cdots,\;D^{1},\;D^{0}$ are of respective ranks $k+1$, $2k+1$,$\cdots$, $sk+1$, $(s+1)k+1$.\\
(iii) Each Cauchy characteristic subdistribution $L(D^j)$ of $D^j$ is a subdistribution of constant corank one in each $D^{j+1}$, for $j=1,\cdots,s-1$, and $L(D^s)=0$.\\
(iv) there exists a completely integrable subdistribution $F\subset D^1$ of corank one in $D^1$.\\



\begin{Rem}\label{covaF} It should be remarked that the covariant subdistribution $F\subset D^1$ is uniquely determined by $D^1$ itself. This covariant subdistribution $F$ is completely described in \cite{Ad} and \cite{ShYa} where it is defined in terms of the annihilating Pfaffian system $(D^{1})^\bot \subset T^*M$ $($\cite{KumpRub}$)$. For a complete clarification on this fact see \cite{MorPel}. \\
\end{Rem}
\begin{Rem}\label{22} In  the following we mean by a special multi-flag distribution all distribution generating a special multi-flag.\end{Rem}

From the definition above, we obtain the following sandwich diagram:\\

 $ \begin{matrix}\label{sandwich}
D^{s}&\subset\hfill D^{s-1}\hfill&\subset\cdots\subset &D^{j}&\subset\cdots\subset&D^{1}        &\subset & D^{0}= TM\hfill\cr
                            \hfill\cup\hfill&\hfill\;\;\cup\hfill& \cdots
                                &\cup    &\cdots                         &\cup\hfill   &              &\cr
                             \hfill L(D^{s-1})&\subset L(D^{s-2})&\subset\cdots\subset  & L(D^{j-1})&\subset\cdots\subset& F\hfill&            &\cr
\end{matrix}$\\

All vertical inclusions in this diagram are of codimension one, while all horizontal inclusions are of codimension $k$. The squares built by these inclusions can be  perceived as certain sandwiches, i.e each “subdiagram"  number $j$ indexed by the upper left vertices $D^{j}$:\\

$ \begin{matrix}\label{sandwich}
D^{j}&\subset\hfill D^{j-1}\hfill\cr
                           \hfill\cup\hfill&\hfill\cup\hfill\cr
                            \hfill L(D^{j-1})&\subset L(D^{j-2})&\cr
\end{matrix}$\\

is called  sandwich number $j$.\\
We can read the length $s$ of the special $k$-flag by adding one to the total number of sandwiches in the sandwich diagram. \\

\begin{Rem}\label{sing} In a sandwich number $j$, at each point $x\in M$, in the $(k+1)$ dimensional vector space $D^{j-1}/L(D^{j-1})(x)$ we can look for the relative position of the $k$ dimensional subspace $L(D^{j-2})/L(D^{j-1})(x)$ and the $1$ dimensional subspace $D^{j}/L(D^{j-1})(x)$:

either  $L(D^{j-2})/L(D^{j-1})(x)\oplus D^{j}/L(D^{j-1})(x)=D^{j-1}/L(D^{j-1})(x)$

or $ D^{j}/L(D^{j-1})(x)\subset L(D^{j-2})/L(D^{j-1})(x)$.\\
We say that $x\in M$ is a {\bf regular point} if  the first situation is true in each  sandwich number $j$, for $j=1,\cdots,s$. Otherwise $x$ is called a {\bf singular point}. \\

The set of singular points in the context of an articulated arm is studied in \cite{Sla} and these results  will  be published in a future paper.
\end{Rem}

\subsection{\bf Special multi-flags and  articulated arm}\label{distdrap}${}$\\
The space  $(\R^{k+1})^{n+2}$, will be written as the product  $\R^{k+1}_0 \times \cdots \R^{k+1}_i \times\cdots \R^{k+1}_{n+1}$.  Let $x_i=(x_i^1,\cdots,x_i^{k+1})$  be the canonical coordinates on the space  $\mathbb{R}^{k+1}_i$ which is equipped with its  canonical  scalar product $<\;,\;>$. $(\R^{k+1})^{n+2}$ is equipped with its canonical scalar product too.\\

Consider an articulated arm of length ($n+1$) denoted  by $(M_0,\cdots M_{n+1})$. In this paper,  we assume that  the distances $l_i$  are all equal to $1$.
 On $(\R^{k+1})^{n+2}$, consider the vector fields:
  \begin{eqnarray}\label{normalvec}
{\cal Z}_i =\dis\sum_{r=1}^{k+1} (x_{i+1}^r-x_{i}^r)\frac{\partial}{\partial x_{i}^r} \textrm{ for }  i=0,\cdots,n
\end{eqnarray}
From our  previous assumptions (see  section \ref{intro}),  the kinematic evolution of the articulated arm is described by a controlled system:
\begin{eqnarray}\label{contsystgene}
\dot{q}=\dis\sum_{i=0}^n u_i{\cal Z}_i+\dis\sum_{r=1}^{k+1} u_{n+r}\dis\frac{\partial}{\partial x_{n+1}^r}
\end{eqnarray}
with the following constraints:\\
$||x_{i}-x_{i+1}||=1$ for $i=0\cdots n$ (see \cite{Li} chapiter 2).\\

Consider the map $\Psi_i (x_0,\cdots,x_{n+1})=||x_{i}-x_{i+1}||^2-1$. Then, the configuration space $\cal C$ is the set
\begin{eqnarray}\label{eqcontraint}
\{ (x_0,\cdots,x_{n+1}), \textrm{ such that } \Psi_i (x_0,\cdots,x_{n+1})=0 \textrm{ for } i=0,\cdots, n \}
\end{eqnarray}

 For $i=0,\cdots,n$, the vector field:
\begin{eqnarray}
{\cal N}_i=\dis\sum_{r=1}^{k+1} (x_{i+1}^r-x_{i}^r)[\frac{\partial}{\partial x_{i+1}^r}- \frac{\partial}{\partial x_{i}^r}]
\end{eqnarray}
is proportional to the gradiant of $\Psi_i$. So the tangent space  $T_q{\cal C}$ is the subspace of $T_q(\R^{k+1})^{n+2}$ which is orthogonal to ${\cal N}_i(q)$ for $i=0,\cdots,n$.\\

Denote by ${\cal E}$ the distribution generated by the vector fields $$\{ {\cal Z}_0,\cdots,{\cal Z}_n,\dis\frac{\partial}{\partial x_{n+1}^1},\cdots,\dis\frac{\partial}{\partial x_{n+1}^{k+1}}\}.$$

 \begin{Lem}\label{distribinduite}
Let $\D$  be the distribution on $\cal C$ defined by $\D(q)=T_q{\cal C}\cap {\cal E}$. Then $\D$ is a distribution of dimension $k+1$  generated by
$$(x_{n+1}^r-x_n^r)[\dis\sum_{i=0}^n\dis\prod_{j=i+1}^{n+1} A_j  {\cal Z}_i]+\dis\frac{\partial}{\partial x_{n+1}^r}  \textrm{ for } r=1\cdots k+1$$
 where $A_j(q)=-<{\cal N}_{j}(q),{\cal N}_{j-1}(q)>=<{\cal Z}_{j}(q),{\cal N}_{j-1}(q)>$ for $j=1,\cdots, n$ and $A_{n+1}=1$.\\
 \end{Lem}

\begin{proof}
Any vector field $X$ tangent to $\cal E$ can be written as:
$$X=\dis\sum_{i=0}^n \l_i{\cal Z}_i+\dis\sum_{r=1}^{k+1} \mu_r\dis\frac{\partial}{\partial x_{n+1}^r}$$
On the other hand, on $\cal C$, a vector fields $X$ is tangent to ${\cal C}$, if and only if $X$ is orthogonal to the vector fields ${\cal N}_0,\cdots,{\cal N}_n$.

\noindent For $i=0,\cdots, n-1$, each relation $<X,{\cal N}_i>=0$ is reduced to $<\l_{i+1}{\cal Z}_{i+1}+\l_i{\cal Z}_i,{\cal N}_i>=0$, which is equivalent to
\begin{eqnarray}\label{relat0n-1}
\l_i=\l_{i+1}A_{i+1}
\end{eqnarray}
Similarly, the relation $<X,{\cal N}_i>=0$ induces
\begin{eqnarray}\label{relan}
\l_n=\dis\sum_{r=1}^{k+1}\mu_r(x_{n+1}^r-x_n^r)
\end{eqnarray}
and from (\ref{relat0n-1}) and (\ref{relan})  we get $\l_i=\dis\prod_{j=i+1}^{n} A_i \l_n$,  for $i=0,\cdots n-1$.\\
\end{proof}

The properties of $\D$ are summarized  in the following result. (see also \cite{Li} chapiter 2))

\begin{The}\label{drap} ${}$
On ${\cal C}$, the distribution  $\D$ satisfies the following properties:
\begin{enumerate}
\item $\D$ is a  distribution of rank $k+1$.

\item The distribution $\D$ is a special $k$-flag  on ${\cal C}$ of length $(n+1)$. \\
\end{enumerate}
 \end{The}

 The first part of Theorem \ref{drap} is a direct consequence of Lemma \ref{distribinduite}. Part (2) will be proved in section \ref{preuveTH1} in terms of hyperspherical coordinates.

\section{\bf The evolution of the articulated arm  in a system of  angular coordinates }\label{systangulaire}
 Given an articulated arm $(M_0,\cdots M_{n+1})$ in $ \mathbb{R}^{k+1}$, we will show that the constraint  controlled system (\ref{contsystgene}) can be written in the same way as (\ref{systcont}) in an adapted system of  angular coordinates  with $(k+1)$ controls, namely $v_n$ (the "normal" velocity of $M_{n+1}$) and the $k$ components of the  "tangential velocity" of $M_{n+1}$  (Theorem \ref{system}).
 \subsection{\bf Hyperspherical coordinates }\label{sphcoord}${}$\\
  The following  map
$$\G(x_0,x_1,\cdots,x_i,\cdots,x_{n+1})= (x_0,x_1-x_0,\cdots,x_i-x_{i-1},\cdots\,x_{n+1}-x_n)$$
  implies a global  diffeomorphism of   $( \mathbb{R}^{k+1})^{n+2}$ into itself and $\G({\cal C})=\R^{k+1}\times(\mathbb{S}^k)^{n+1}$ where  $\mathbb{S}^k$  is the canonical sphere  in $ \mathbb{R}^{k+1}$. In this representation, the canonical coordinates on  $ (\mathbb{R}^{k+1})^{n+2}=\G((\mathbb{R}^{k+1})^{n+2})$ will be  denoted by $(x_0,z_1,\cdots,z_i,\cdots,z_{n+1})$ so that  $\G$ is given by $x_0=x_0$ and $z_i=x_{i+1}-x_i$ for $i=0,\cdots n$.
Via this  global chart, each point $q=(x_0,x_1,\cdots,x_i,\cdots,x_{n+1})\in{\cal C}$ can  be identified with $(x_0,z_1,\cdots,z_i,\cdots,z_{n-1})\in \mathbb{R}^{k+1}\times(\mathbb{S}^k)^{n+1}$  for $i=0,\cdots, n$ and ${\cal C}$ can be identified with ${\cal S}= \mathbb{R}^{k+1}\times(\mathbb{S}^k)^{n+1}$. \\

We will put on $(\mathbb{S}^k)^{n+1}$ charts given by  {\it hyperspherical coordinates}.  We first recall some basic facts about this type of coordinates.\\

The {\it  hyperspherical coordinates } in $\R^{k+1}$  are given by the  relations:

 \begin{eqnarray}\label{cordonnees}
\begin{cases}z^1=\r\phi^1(\theta)= \r\sin {\theta^1} \cdots\sin{\theta^{k-1}}\sin{\theta^k} \cr
z^2=\r\phi^2(\theta)=\r\sin {\theta^1} \cdots\sin{\theta^{k-1}}\cos{\theta^k}\cr
z^3=\r\phi^3(\theta)=\r\sin {\theta^1} \cdots\sin{\theta^{k-2}}\cos{\theta^{k-1}} \cr
\qquad\cdots\cr
z^{k}=\r\phi^k(\theta)=\r\sin{ \theta^1} \cos{\theta^{2} }\cr
z^{k+1}=\r\phi^{k+1}(\theta)=\r\cos{\theta^1}\cr\end{cases}
\end{eqnarray}

with $\r^2=(z^1)^2+\cdots+(z^{k+1})^2$,  $0\leq \theta^k\leq 2\pi$
and $0\leq\theta^j\leq \pi$ for $1 \leq j \leq k-1$.

We consider  $\hat{\Phi}(\r,\theta)=\r\Phi(\theta)=z$, the application from  $]0,+\infty[\times[0,\pi]\times\cdots \times [0,\pi]\times[0,2\pi]$ to $\R^{k+1}$.

\begin{Rem}  The previous expression uses the "geographical" version of hyperspherical coordinates. An another version, maybe more usual, can be obtained by taking  $\displaystyle\frac{\pi}{2}-\theta^k$ instead of $\theta^k$ and then, permuting the functions sine and cosine in each formula. However, our choice is motivated by the following  fact:  the evolution of the articulated arm of length  $(n+1)$   written in a such chart, (see (\ref{systemcont}) ) gives exactly the system (\ref{systcont}) for $n=1$.
\end{Rem}

The jacobian matrix $D\hat{\Phi}$ of $\hat{\Phi}$ is:

\begin{eqnarray}\nonumber
\begin{pmatrix}
\sin {\theta^1} \cdots\sin{\theta^{k-1}}\sin{\theta^k}&\r\cos {\theta^1} \cdots\sin{\theta^{k-1}}\sin{\theta^k}&\cdots&\r\sin {\theta^1} \cdots\sin{\theta^{k-1}}\cos{\theta^k}\cr
\sin {\theta^1} \cdots\sin{\theta^{k-1}}\cos{\theta^k}&\r\cos {\theta^1} \cdots\sin{\theta^{k-1}}\cos{\theta^k}&\cdots&-\r\sin {\theta^1} \cdots\sin{\theta^{k-1}}\sin{\theta^k}\cr
\cdots&\cdots&\cdots&\cdots\cr
\cos{\theta^1}& -\r\sin{\theta^1}&0&0\cr
\end{pmatrix}
\end{eqnarray}

It is well known that  $\det(D\hat{\Phi})(\r,\theta)=(-1)^{[k+1/2]}(\r)^k\dis\prod_{i=1}^{k-1}(\sin{\theta^{k-i}})^i$. It follows that  $D\hat{\Phi}$ is  invertible only for  $0\leq \theta^k\leq 2\pi$ and $0< \theta^j< \pi$ for  $j=1,\cdots,k-1$.\\

In the sequence we note  $\{\nu,\Theta^1,\cdots,\Theta^k\}$  the moving frame  on  \\$\hat{\Phi}(]0,+\infty[\times[0,\pi]\times\cdots \times]0,\pi[\times]0,2\pi[)$ that is  the image, by  $D\hat{\Phi}$, of the canonical frame
 $\{\dis\frac{\partial}{\partial t},\dis\frac{\partial}{\partial \theta^1},\cdots,\frac{\partial}{\partial \theta^k}\}$.\\

Consider a  point $z=\Phi(\theta)=\hat{\Phi}(1,\theta)$. We note that, in this case, we have   $$D\hat{\Phi}=
\begin{pmatrix}\label{jacobian}
\phi&\frac{\partial \phi}{\partial \theta^1}&\cdots &\frac{\partial \phi}{\partial \theta^k}\cr
\end{pmatrix}$$
where $\phi$ (resp. $ \frac{\partial \phi}{\partial \theta^j}$) is
the column vector of components $\{\phi^1,\cdots,$ $\phi^{k+1}\}$\\ (resp. $\{\frac{\partial \phi^1}{\partial\theta^j},\cdots,\frac{\partial \phi^{k+1}}{\partial \theta^j}\}$).\\
Note that,  in canonical coordinates,  this moving frame  $\{\nu,\Theta^1,\cdots,\Theta^k\}$   can be written
 \begin{eqnarray}\label{cononmoving}
\n(z)=\dis\sum_{r=1}^{k+1}\phi^r\dis\frac{\partial}{\partial x^r}\;\;\;\; \textrm{ and }\;\;\;\; \Theta^j(z)=\dis\sum_{r=1}^{k+1}\frac{\partial{\phi}^r}{\partial \theta^j}\frac{\partial}{\partial x^r}
\end{eqnarray}
and in hyperpsherical coordinates this moving frame is:
 \begin{eqnarray}\label{hypermoving}
\n(\theta)=\dis\frac{\partial}{\partial  t}\;\;\;\; \textrm{ and }\;\;\;\; \Theta^j(\theta)=\dis\frac{\partial}{\partial \theta^j}
\end{eqnarray}

 On the other hand, it is well known that  $\det(D\hat{\Phi})(\r,\theta)=(-1)^{[k+1/2]}(\r)^k\dis\prod_{i=1}^{k-1}(\sin{\theta^{k-i}})^i$. It follows that  $D\hat{\Phi}$ is  invertible only for  $0\leq \theta^k\leq 2\pi$ and $0< \theta^j< \pi$ for  $j=1,\cdots,k-1$.\\
The  column vectors of the  jacobian matrix $D\hat{\Phi}$ are  pairwise orthogonal  and  we have

$||\phi ||^2=||\frac{\partial \phi}{\partial \theta^1}||^2=1$, $||\frac{\partial \phi}{\partial \theta^j}||^2=(\sin {\theta^1} \cdots\sin{\theta^{j-1}})^2$ for $j=2,\cdots,k$.\\
So, at each point $z=\Phi(\theta)$, the inverse of this jacobian  matrix is  the transpose of the matrix:
\begin{eqnarray}\label{inversejacobienne}
\dis\begin{pmatrix}
\phi&\dis\frac{1}{||\frac{\partial \phi}{\partial \theta^1}||}\frac{\partial \phi}{\partial\theta^1}&\cdots
&\dis\frac{1}{||\frac{\partial\phi}{\partial \theta^k}||}\frac{\partial \phi}{\partial \theta^k}\cr
\end{pmatrix}
\end{eqnarray}
This inverse  exists only if $0\leq \theta^k\leq 2\pi$ and $0< \theta^j< \pi$ for $j=1,\cdots,k-1$.\\

Consider a  hyperspherical chart   $y=\hat{\Phi}(\r,\theta)+a$ around some $a\in \R^{k+1}$
 and   note $\{\n,\Theta^1,\cdots, \Theta^k\}$ its associated moving frame.
 
 Let $y'=\hat{\Phi}'(\r',\theta')+a'$ be another  hyperspherical chart around some other point $a'\in \R^{k+1}$ such that its domain intersects the domain  of $\hat{\Phi}$.
  Note $\{\n',{\Theta'}^1,\cdots, {\Theta'}^k\}$ the moving frame  associated to it. So, at each point of the sphere  $\mathbb{S}_a$ of center $a$ of radius $1$ in $\R^{k+1}$ contained in  the intersection of these  domains, we can write :
\begin{eqnarray}\label{relentrenorm}
\nu'=A\nu+\dis\sum_{j=1}^k B^j \Theta^j
\end{eqnarray}

The components of these vectors are actually the components of the first column vector  of the matrix $[D\hat{\Phi}]^{-1}\circ D\hat{\Phi}'$.  According to  (\ref{inversejacobienne}), we get:
\begin{eqnarray}\label{chgtbase}
\begin{cases}
\bullet  A(\theta,\theta')=\dis\sum_{r=1}^{k+1}{\phi}^r{\phi'}^r  \\
\bullet   B^1(\theta,\theta')=\dis\sum_{r=1}^{k+1}\frac{\partial{\phi^r}}{\partial {\theta^1}}{\phi'}^r  \\
\bullet  B^j(\theta,\theta')=\dis\frac{1}{||\frac{\partial{\phi}}{\partial {\theta}^j}||}\dis\sum_{r=1}^{k+1}\frac{\partial{\phi}^r}{\partial {\theta}^j}{\phi'}^r
\end{cases}
\end{eqnarray}

 \begin{Rem}\label{projnu}${}$\\
 the vector $\dis\sum_{j=1}^k B^j \Theta^j$ is  nothing but the orthogonal projection of $\n'$ on to $T\mathbb{S}_a$.
\end{Rem}
\bigskip
\subsection{The evolution of the articulated arm on ${\cal S}$}\label{resultangsyst}${}$\\
 Coming back  to   ${\cal S}= \mathbb{R}^{k+1}\times(\mathbb{S}^k)^{n+1}$ which is considered as a subset in $(\R^{k+1})^{n+2}$,  let  $\mathbb{S}_{i}$, $i=0,\cdots, n$, be the canonical sphere in $\R^{k+1}_{i+1}$. Recall that the canonical coordinates on $\R^{k+1}_{i}$ are denoted by $z_i=(z_i^1,\cdots, z_i^r,\cdots, z_i^{k+1})$.  Given   a point $\a$ in the sphere $\mathbb{S}_{i}$,  there exists a  hyperspherical chart $z_{i+1}=\hat{\Phi}_{i}(\r_{i},\theta_{i})=\r_{i}\Phi_i(\theta_i^1,\cdots\theta_i^k)$ defined for
 $0\leq \theta_i^k\leq 2\pi$ and $0<\theta_i^j<\pi$, $j=1,\cdots,k-1$,
 where $\Phi_i(0,\cdots,0)=\a$. So, for a given point $q=(x_0,z_1,\cdots,z_i,\cdots,z_{n+1})\in {\cal S}$, we get a chart $(Id-x_0,\hat{\Phi}_0,\cdots,\hat{\Phi}_1,\cdots,\hat{\Phi}_{n})$
 centered at $q$,  such that its restriction to $\r_i=1$, $i=0,\cdots,n$,  induces a chart of
 $\cal S$ (centered at $q$).
For $ i=0,\cdots,n$, in a neighborhood of each  $z_{i+1}\in \R^{k+1}_{i+1}$, we consider the moving frame\\
${\cal R}_i=\{\nu_i,\Theta_i^1,\cdots,\Theta_i^k\}$ . \\

\begin{Rem}\label{sphere plongee} ${}$
\begin{enumerate}
\item Given $q=(x_0,\cdots,x_{n+1})\in {\cal C}$, for  $i=0 ,\cdots,n$,  denote by $\tilde{\mathbb{S}}_i$ the sphere in $\R^{k+1}$ of center $x_{i}$ and radius $1$. One can put on $\R^{k+1}$ the hyperspherical coordinates $y_i=\hat{\Phi}_{i}(\r_{i},\theta_{i})+x_{i}$. As $x_{i+1}$ belongs to $\tilde{\mathbb{S}}_i$, on a neighborhood of $x_{i+1},$ we have also the following moving frame (again denoted ${\cal R}_i$):
$${\cal R}_i=\{\nu_i, \Theta_i^1,\cdots,\Theta_i^k\}.$$
Note that on $x_{i+1}$,  the outward normal unit vector of $\tilde{\mathbb{S}}_i$ is  $\n_i(x_{i+1})$,  and\\ $\{ \Theta_i^1( x_{i+1}),\cdots, \Theta_i^k( x_{i+1})\}$ is a basis $T_{x_{i+1}}\tilde{\mathbb{S}}_i$.\\
\item  From part (1), if $\zeta=\G(q)$,  we get an isomorphism from  $T_{\zeta}{\cal S}$ to \\$T_{x_0}\R^{m+1}\oplus T_{x_1} \tilde{\mathbb{S}}_0\oplus \cdots\oplus T_{x_{n+1}}\tilde{\mathbb{S}}_{n}$.\\ So we can identify $T_\zeta{\cal S}$ with  $T_{x_0}\R^{n+1}\oplus T_{x_1} \tilde{\mathbb{S}}_0\oplus \cdots\oplus T_{x_{n+1}}\tilde{\mathbb{S}}_{n}$.
\end{enumerate}
\end{Rem}
\bigskip
\begin{Not}\label{nota2} ${}$\\ in hyperspherical coordinates, we define on ${\cal S}$:
\begin{enumerate}
\item $  A_i=\dis\sum_{r=1}^{k+1}\phi_{i-1}^r\phi_{i}^r\qquad$
  for $i=1,\cdots,n$  and $A_{n+1}=1$;\\
\item $Z_0=\dis\sum_{r=1}^{k+1}\phi_0^r\frac{\partial}{\partial x^r}$\\
\item $Z_i=\dis\sum_{j=1}^k B_i^j\frac{\partial}{\partial{\theta_{i-1}^j}}$ for $i=1,\cdots,n$\\
 with:\\
$\bullet$ $ B^1_i=\dis\sum_{r=1}^{k+1}\frac{\partial{\phi_{i-1}^r}}{\partial {\theta_{i-1}^1}}\phi_i^r\qquad$
for $i=1,\cdots,n$\\
$\bullet$   $B_i^j={\frac{1}{||\frac{\partial\phi_{i-1}}{\partial \theta_{i-1}^j}||}}\dis\sum_{r=1}^{k+1}\dis\frac{\partial{\phi_{i-1}^r}}{\partial {\theta_{i-1}^j}}\phi_i^r\qquad$ for $i=1,\cdots,n$ and  $j=2,\cdots,k$\\

\item $X^i_m=\dis\frac{\partial}{\partial \theta^i_m}$,  for $i=1,\cdots,k$,  $m=0,\cdots, n$\\

\item$X_m^0=\dis\sum_{i=0}^m f^i_mZ_i$, for $m=0,\cdots,n$

with  $f^r_m=\dis\prod_{j=r+1}^{m}A_j$,  for
$r=0,\cdots,m-1$,  and $f_m^m=1$   for $m=0,\cdots,n$.\\

\item $\D_n$  the distribution  generated by $\{X^0_n,X^1_n,\cdots,X^k_n\}$ (with previous notations).
\end{enumerate}
\end{Not}

\begin{Rem}\label{formeconcise}${}$
\begin{enumerate}
\item For $i=0,\cdots,n$, consider $[\Phi_i]$ the column matrix of components $(\phi_i^1,\cdots,$ $\phi_i^{k+1})$ and  $[D\Phi_i]^{-1}$ the matrix composed by the
  last $k$  rows of the jacobian matrix of the application  $(\Phi_i)^{-1}$.  Finally, denote by  $[\dis\frac{\partial}{\partial\theta_i}]$ $($resp. $[\dis\frac{\partial}{\partial x}])$ the column matrix  of components $(\frac{\partial}{\partial{\theta_i^1}},\cdots,\frac{\partial}{\partial{\theta_i^k}})$, $($resp.$(\frac{\partial}{\partial{x^1}},\cdots,\frac{\partial}{\partial{x^{k+1}}}))$
 for $i=0,\cdots,n$. So we can write:\\
 ${}\qquad \qquad \qquad Z_0=^t[\dis\frac{\partial}{\partial x}]\;[\Phi_0]$   and $Z_i=^t[\dis\frac{\partial}{\partial\theta_{i-1}}]{[[D\Phi_{i-1}]]^{-1}}[\Phi_i]$,  for $i=1,\cdots,n.$\\
\item According  to Remark  \ref{sphere plongee} part (2), in the identification :\\
${}\;\;\;\;\;\;\;\;\;\;\;T_\zeta{\cal S}\equiv T_{x_0}\R^{n+1}\oplus T_{x_1} \tilde{\mathbb{S}}_0\oplus \cdots\oplus T_{x_{n+1}}\tilde{\mathbb{S}}_{n},$\\ 
the component of $X_n^0$ on $T_{x_i}\tilde{\mathbb{S}}_{i-1}$ is exactly $f_n^iZ_i(\zeta)$.
\end{enumerate}
\end{Rem}
With the previous  notations we have the following result

\begin{The}\label{system}${}$
\begin{enumerate}
\item On $\cal S$,  the distribution $\D_n$ is the image by $\G$ of the distribution $\D$  \\
where $\G:{\cal C}\ap {\cal S}$ is the diffeomorphism defined at the beginning of Subsection \ref{sphcoord}.

\item The evolution of the articulated arm of length  $(n+1)$  is described in a chart, by the following   controlled system with $k+1$ controls :
\begin{eqnarray}\label{systemcont}
\begin{cases}\dot{x}^1= v_0\phi_0^1\cr
\dot{x}^2=v_0\phi_0^2\cr
\qquad\cdots\cr
\dot{x}^{k+1}=v_0\phi_0^{k+1}\cr
\dot{\theta}_{0}^1=v_1B_1^1\cr
\qquad\cdots\cr
\dot{\theta}_{0}^k=v_1B_1^k\cr
\qquad\cdots\cr
\dot{\theta}_{i}^1=v_{i+1}B_{i+1}^1\cr
\qquad\cdots\cr
\dot{\theta}_{i}^k=v_{i+1}B_{i+1}^k\cr
\qquad\cdots\cr
\dot{\theta}_{n-1}^1=v_nB_{n}^1\cr
\qquad\cdots\cr
\dot{\theta}_{n-1}^k=v_nB_{n}^k\cr
\dot{\theta}_n^1=v_{\theta_n^1} \cr
\qquad\cdots\cr
\dot{\theta}_n^k=v_{\theta_n^k}\cr\end{cases}
\end{eqnarray}
\noindent where $v_i=v_n\dis\prod_{r=i+1}^{n}A_r$ and  $(v_{\theta_n^1},\cdots,v_{\theta_n^k}, v_n)$ are the $(k+1)$ controls of the system (\ref{systemcont}).\\

Moreover, according to Remark \ref{sphere plongee}   we have:

$\bullet$ $(v_{\theta_n^1},\cdots,v_{\theta_n^k})$ are the "tangential components"  of the velocity  of $M_{n+1}$,  namely  the components, in the canonical basis of   $T_{x_n+1}\tilde{\mathbb{S}}_{n}$, of the orthogonal projection of the velocity of $M_{n+1}$;

$\bullet$   $v_{i-1}$ is  the "normal  velocity" of  $M_i$ for all  $i =1,\cdots,n+1$, namely the components of the orthogonal projection of the velocity of $M_i$ on the direction generated by $\n_{i-1}(x_i)$.
\end{enumerate}
\end{The}

\begin{Rem}
Equations of system $(\ref{systemcont})$, for $k=1$, are  exactly $($with the same  notations$)$ the classical modeling of the car with $n$ trailers $($
 \cite{FLIE}, \cite{Jea}, \cite{Lau}, 
 \cite{PasResp3}, \cite{Sor1}$)$ 
 \end{Rem}
\begin{Rem}\label{singarm}
 According to Remark \ref{sing} , for $k\geq 2$, a point $q=(x_0,z_1,\cdots,z_{n+1})$ is singular if and only if there  exists an  index $0\leq i\leq n$ such that $A_i(q)=0$ which is equivalent to $[M_{i-1},M_i]$  and $[M_i,M_{i+1}]$ are orthogonal in $M_i$. In this situation,  the velocity all  $M_j$  points are zero, for $j< i$. The set of such points is studied in \cite{Sla}.
\end{Rem}

\section{Proof of Theorem \ref{system}}

For part (1), it is sufficient to prove that any germ  curve $\g$ in ${\cal C}$ is tangent to ${\D}$ at a point $q\in{\cal C}$, if and only if,  $\G \circ \g$ is  tangent to the distribution generated by $\{X^0_n,X^1_n,\cdots,X^k_n\}$  at  $ {\G}(q)$.\\

Let be $\g(t)=(x_0(t),x_1(t),\cdots,x_{n+1}(t))$ a curve in ${\cal C}$  defined on $]-\varepsilon,\varepsilon[$ with $\g(0)=q$. We have
\begin{eqnarray}\label{decompositionxi}
x_i=x_{i-1}+z_i\textrm{ for } i=1\cdots n+1
\end{eqnarray}
Assume that $\g$ is tangent to ${\D}$. It follows that we have
 for each $t$  :
\begin{eqnarray}\label{contvitesse}
\dot{x}_{i-1}=v_{i-1}z_{i}, \textrm{ for } i=1,\cdots,n+1,
\end{eqnarray}
In view of (\ref{contvitesse}), by differentiation of (\ref{decompositionxi}) we get 
\begin{eqnarray}\label{decompovitesse}
\dot{x}_i=v_{i-1}z_i+ \dot{z}_i \textrm{ for }  i=1,\cdots, n+1
\end{eqnarray}
for some $v_{i}\in \R$ (which depends of $t$).

On the other hand, in $\R^{k+1}$ we  also have the orthogonal decomposition
\begin{eqnarray}\label{decompositonaiai+1}
z_{i+1}=<z_{i+1},z_{i}>z_i+\tilde{z}_{i}\;\;\;\;\;\;\;\; i=0,\cdots, n.
\end{eqnarray}
 where 
 \begin{eqnarray}\label{orthogonal}
\tilde{z}_{i} \textrm{ is the   othogonal  projection of  } z_{i+1} \textrm{  on to the hyperplan orthogonal to } z_{i}.
\end{eqnarray}

So we get:
\begin{eqnarray}\label{decompoxi}
\dot{x}_i= v_{i} z_{i+1}= v_{i}<z_{i+1},z_i>z_i+v_{i} \tilde{z}_i\textrm{ for } i=1,\cdots,n
\end{eqnarray}
Comparing (\ref{decompovitesse}) and (\ref{decompoxi}), we get :
\begin{eqnarray}\label{relat v-i}
v_{i-1}=v_{i}<z_{i+1},z_i>\;\textrm{ and } \;\;\dot{z}_i=v_{i} \tilde{z}_{i}  \textrm{ for } i=1,\cdots n
\end{eqnarray}

\noindent $\bullet$  for $i=n+1$,  we have the orthogonal decomposition
\begin{eqnarray}\label{decompoxh}
\dot{x}_{n+1}=v_{n+1}z_{n+1}+ \tilde{z}_{n+1}
\end{eqnarray}

But  we also have
$$\dot{x}_{n+1}= \dot{z}_{n+1}+\dot{x}_{n}=v_{n}z_{n+1}+\dot{z}_{n+1}$$

So from the unicity of the orthogonal decomposition,  we get 
\begin{eqnarray}\label{decompvh}
v_n=v_{n+1} \textrm{ and } \dot{z}_{n+1}=\tilde{z}_{n+1}
\end{eqnarray}
\noindent $\bullet$ for $i=0$  we have of course 
\begin{eqnarray}\label{decomposx0}
\dot{x}_0=v_0 z_1
\end{eqnarray}

From (\ref{relat v-i}) and (\ref{decompvh}) we get the relation
\begin{eqnarray}\label{calcul vi}
v_i=\dis\prod_{j=i+2}^{n}<z_{j+1},z_j>v_n \textrm{ for } i=0\cdots, n-1
\end{eqnarray}

On the other hand,
according to Remark \ref{sphere plongee} part (2), as  $x_{i}\in \tilde{\mathbb{S}}_{i-1}$,  for $t=0$,  $\dot{x}_i$  can be  considered as a vector in $ T_{x_i}\R^{k+1}$ and so, in the moving frame ${\cal R}_i$, the relations (\ref{decompoxi}), (\ref{calcul vi}) and  (\ref{decompoxh})   give rise to the following decompositions
\begin{eqnarray}\label{decompovitess2}
\dot{x}_i=v_{i-1}\n_{i-1}(x_{i})+v_{i}{Z}_{i}  \textrm{ for } i=1,\cdots n, \textrm{ and } \dot{x}_{n+1}=v_n\n_n(x_{n+1})+\tilde{Z}_{n+1}
\end{eqnarray}

where ${Z}_{i}$ (resp. $\tilde{Z}_{n+1}$) belongs to  $ T_{x_i}\tilde{\mathbb{S}}_{i-1}$ pour $i=1,\cdots, n$ (resp. $ T_{x_{n+1}}\tilde{\mathbb{S}}_{n}$).\\

 In fact,  according to (\ref{orthogonal}) and   Remark \ref{sphere plongee} part (2), ${Z}_{i}$ is the orthogonal projection\\ of $(D\hat{\Phi}^{-1}_{i-1})^{-1}\circ D\hat{\Phi}_{i}(\n_{i}(x_{i+1}))$ on $ T_{x_i}\tilde{\mathbb{S}}_{i-1}$, for $i=1,\cdots, n$ . Taking in account  Remark \ref{projnu} and  (\ref{chgtbase}),   we get the expression  of ${Z}_{i}$ given in Notations \ref{nota2} for $i=1,\cdots n$ . Moreover,  the  $i^e$ component $\dot{z_i}$  of  $D\G(\dot{\g}(0))$ is $v_{i-1} Z_i$ in hyperspherical coordinates for $i=1,\cdots n$ . On the other hand, in the hyperspherical  $\Phi_n$, the vector  $\tilde{Z}_{n+1}$ can be written
 $$\tilde{Z}_{n+1}=\dis\sum_{j=1}^k w_j\dis\frac{\partial}{\partial \theta_n^j}$$ 
Finally, from (\ref{calcul vi}) and the value of $A_j$, we get:
$$v_i=f^i_n v_n \textrm{ for } i=0,\cdots n$$

On the other hand, from (\ref{decomposx0}),  as vector in $T_{x_0}\R^{m+1}$,  we have the following decomposition:
$$\dot{x}_0=v_0\dis\sum_{r=1}^m\phi_0^r\dis\frac{\partial}{\partial x^r_0}$$

So, according to Remark \ref{sphere plongee} part (2) and Remark \ref{formeconcise} part (2), in hyperspherical coordinates,  we finally  obtain  that, $D\G(\dot{\g}(0))$ is tangent $\D_n$ at $\zeta=\G(q)$.\\

Conversely, let be $\d(t)=(x_0(t),z_1(t),\cdots,z_{n+1}(t))$ a curve  in  ${\cal S}$, defined on $]-\varepsilon,\varepsilon[$, and  such that $\d(0)=\zeta=\G(q)$. Assume that, in hyperpsherical coordinates,  $\d$ is tangent  to $\D_n$. and so  $\d$ satisfies the differential equation (\ref{systemcont}). According to the definitions of $Z_i$  in   hyperpsherical coordinates and  taking in account (\ref{orthogonal}) and   Remark \ref{sphere plongee} part (2),  there exists   a curve $v_n(t)$ in $\R$ such that,
for any $i=1,\cdots n$, we have 

$\dot{z}_{i}=[(\dis\prod_{j=i+1}^{n}<z_{j+1},z_j>)v_n ][z_{i+1}-<z_{i+1},z_{i}>z_{i}] \textrm{ for }  i=1,\cdots n$

 and 

\begin{eqnarray}\label{x0}
\dot{x}_0=[(\dis\prod_{j=1}^{n}<z_{j+1},z_j>)v_h] z_1 .
\end{eqnarray}

It follows that for $ i=1,\cdots n$ we have:
\begin{eqnarray}\label{xi}
\dot{x}_i=\dot{x}_0+\dis\sum_{j=1}^i \dot{z}_j=(\dis\prod_{l=i+1}^{n}<z_{j+1},z_j>v_n) z_{i+1}
\end{eqnarray}

so, according to  (\ref{x0}) and (\ref{xi}), we obtain that  $ \g(t)=\G^{-1}\circ \d(t)$ is a curve in ${\cal C}$ which is tangent to $\D$.\\
\bigskip

 Taking in account part (1) of Therorem  \ref{system}, the kinematic  evolution  of the articulated arm is  a controlled system on $\cal S$ which is exactly (\ref{systemcont}). However, for the completeness of the proof of this result, we must prove the interpretation of the control in terms of the components of the velocity of $M_{i}$, $i=1,\cdots, n+1$.\\
 
\noindent  From (\ref{decompovitess2}), according to  Remark \ref{sphere plongee} part (2), in hyperpsherical coordinates $\Phi_n$ around   $x_{n+1}\in \R^{k+1}$,  the velocity $\dot{x}_{n+1}$ of $M_{n+1}$  can be written:
$$\dot{x}_{n+1}=v_n\n_n(x_{n+1})+\dis\sum_{j=1}^kv_{\theta_n^j}\dis\frac{\partial}{\partial \theta_n^j}$$
Where $v_n$ is the "normal velocity" of $M_{n+1}$ and $(v_{\theta_n^1},\cdots,v_{\theta_n^k})$  are the ”tangential components” of the velocity of $M_{n+1}$. Moreover, according to the value of $X_n^0$, the contol parameters are exactly $(v_n,v_{\theta_n^1},\cdots,v_{\theta_n^k})$.\\
 In the same way, for any $i=1,\cdots,n$, in hyperspherical coordinates $\Phi_{i-1}$  around $x_i\in \R^{k+1}$ the velocity $\dot{x}_i$ of $M_i$ has the decomposition
 $$\dot{x}_i=v_{i-1}\n_{i-1}(x_{i})+v_{i}{Z}_{i}$$
 where $Z_i$ belongs to $T_{x_i}\tilde{\mathbb{S}}_{i-1}$. So, $v_{i-1}$ the "normal velocity " of $M_i$

\section{ Proof  of Theorem \ref{drap}}\label{preuveTH1}${}$

We will see that the  distribution $\D_{n}$ generates actually a
special  $k$-flag of length $(n+1)$ on a  $k(n+2)+1$ dimensional manifold. Let's introduce the  following notations:

\begin{itemize}
\item $\D_{m}$ is the distribution generated by $\{X_m^0,X_m^1,\cdots X_m^k\}$ for $m=1,\cdots,n;$
\item $D^{m+1}$  is the distribution generated by  $X^0_{m}$ and
$\{X^1_j,\cdots,X^k_j \;\; m\leq j\leq n\}$\\   for $m=0,\cdots,n$;
\item $D^{0}=TM$
\item $E^{m+1}$  is the distribution generated by $\{X^1_j,\cdots,X^k_j, \;\; m\leq j\leq n\}$\\ for  $m=0,\cdots,n$.\\
\end{itemize}
\begin{Pro}\label{speflag}${}$
$\Delta_{n}$ is a special k-flag distribution.
 More precisely, it satisfies the following properties:
\begin{enumerate}
\item  For $m=1,\cdots,n+1$, the distributions $D^m$ and $E^m$,  are of respective constant dimensions  $(n-m+2)k+1$ and $(n-m+2)k$;

 \item for $m=1,\cdots,n+1$, $E^{m}$ is an involutive subdistribution of $D^{m}$ of codimension $1$. Moreover $[E^{m+1},D^{m+1}] \subset D^{m}$ for $m=1,\cdots,n$. Actually  $E^{m+1}$  is the "Cauchy-characteristic distribution" of  $D^{m}$ for $m=1,\cdots,n$ $($\cite{Mor5}, \cite{Mor6}$)$;

\item $[D^{m+1},D^{m+1}]=D^{m}$ for all $m=0,\cdots,n$;

\item ${}$\\
$\begin{matrix}
\Delta_{n}=D^{n+1}&\subset\cdots\subset &D^{m}&\subset\cdots\subset&D^{1}        &\subset & D^{0}= TM\hfill\cr
                             \hfill\cup&\cdots                            &\cup    &\cdots                         &\cup\hfill   &           &\cr
                             \hfill E^{n+1}&\subset\cdots\subset  & E^{m}&\subset\cdots\subset& E^{1}\hfill&            &\cr
\end{matrix}$
\end{enumerate}
\end{Pro}

\begin{proof}[ Proof of  Proposition \ref{speflag}]${}$

It is sufficient to show the property (4). The inclusions $[E^{m+1},D^{m+1}]\subset D^{m}$ for $m=n,\cdots,0$,  are an easy consequence of (4) and the properties (1), (2) and  (3) are always true, according to the definition  of spaces $E^m$, $D^m$ and $\D_m$.\\

Denote by  $\D_{0}$ the distribution  generated by  $\{\dis\frac{\partial}{\partial x_0^1},\cdots,\frac{\partial}{\partial x_0^{k+1}}\}$.\\ For all $m= 1, \cdots,n+1$, we have :
$$D^{m}=E^{m+1}\oplus \D_{m-1 }=D^{m+1}+\D_{m-1}.$$
$[D^{m+1},D^{m+1}]$ contains the space generated by $D^{m+1}$ and
the Lie brackets $[X^{i}_{m},X^{0}_{m}]$, for $i=1,\cdots, k$. We will show by
 induction that, in fact, they are  generating
 $[D^{m+1},D^{m+1}] $ modulo $D^{m+1}$. \\

For all $m=n,\cdots,0$, we have  $X^0_{m}=A_{m}X^0_{m-1}+Z_{m}$. It results from the  definition of $A_i$, $Z_i$ and $X^0_m$ that  $[X^i_{m},X^0_{m-1}]=0$. So we have
$$[X^{i}_{m},X^{0}_{m}]=X^i_{m}(A_{m}) X^0_{m-1}+[X^i_{m},Z_{m}].$$
For $j=1,\cdots k+1$, consider the vector fields:\\
 ${}\qquad\qquad Y^j_{m-1}=\phi_{m-1}^j X^0_{m-1}+\dis\sum_{r=1}^k {\frac{1}{||\frac{\partial\phi_{m-1}}{\partial\theta_{m-1}^r}||}}\frac{\partial{\phi_{m-1}^j}}{\partial
 {\theta_{m-1}^r}}X^r_{m-1}$ \\  If we set
  $\hat{\Phi}_{m-1}(\r,\theta_{m-1})= \r \Phi_{m-1}(\theta_{m-1})$,  then we have the relation:
  $$[Y_{m-1}]=[D\hat{\Phi}_{m-1}]^{-1} [X_{m-1}]$$
 where the vectors column $[Y_{m-1}]$  and  $[X_{m-1}]$ have $\{Y^1_{m-1}, \cdots,Y^{k+1}_{m-1}\}$ and\\ $\{X^0_{m-1},\cdots,$ $X^k_{m-1}\}$  as components respectively. It results that
  $\{Y^1_{m-1}, \cdots,Y^{k+1}_{m-1}\}$ is a basis of  $\D_{m-1}$.

 For $m=0,\cdots,n$, we note $[D\hat{\Phi}_m]$ the jacobian matrix of $\hat{\Phi}_m(\r,\theta_m)=\r\Phi_m(\theta_m)$. \\

The following decompositions occur:\\
$X_{m}^0=\dis\sum_{j=1}^{k+1}\phi_m^jY_{m-1}^j$\\
$[X^{i}_{m},X^{0}_{m}]=\dis\sum_{j=1}^{k+1}\frac{\partial{\phi_{m}^j}}{\partial {\theta_{m}^i}}Y_{m-1}^j$ for all  $m=1,\cdots,n$\\

By similar way for $D\hat{\Phi}_m$, we can show that the family of vector fields $$\{X_{m}^0,[X^{1}_{m},X^{0}_{m}],\cdots,[X^{k}_{m},X^{0}_{m}]\}$$ is also a  basis of $\D_{m-1}$. This result is also true for  $m=0$.\\
Since $D^{m+1}=E^{m+2}\oplus\D_{m}$, the space  $[D^{m+1},D^{m+1}]$ contains
$E^{m+2}$, all  vectors\\ $X_{m}^0,X_{m}^1,\cdots,$ $X_m^k$ and the Lie brackets $ [X^{1}_{m},X^{0}_{m}],\cdots,[X^{k}_{m},X^{0}_{m}]$. Also, all the Lie brackets $[X_r^j,X_m^0]$ are zero for
$r=m+1,\cdots,n$ and $j=1,\cdots,k$ since $X_m^0$ does not  depend on variables $\theta_r^j$, for $r=m+1,\cdots,n$ and $j=1,\cdots,k$.
The other Lie brackets $[X_r^j,X_m^i]$  are zero for $r=m+1,\cdots,n$ and $i,j=1,\cdots,k$.\\
Since  $\{X_{m}^0,[X^{1}_{m},X^{0}_{m}],\cdots,[X^{k}_{m},X^{0}_{m}]\}$ is a  basis of $\D_{m-1}$, then we have  $$[D^{m+1},D^{m+1}]=E^{m+1}\oplus\D_{m-1}=D^{m}$$
which completes the proof of  proposition \ref{speflag} and Theorem \ref{drap}.
\end{proof}

\begin{Com}${}$\\ Given two integers $p$ and $m$ such that  $1\leq p<m\leq n$, we can look for the motion of a "sub-induced arm", which consists of segments of the original arm between $M_{p-1}$ and  $M_{m+1}$ included.
We can then study the motion of  $M_{p-1}$ as the motion of the extremity of  this sub arm for the motion commanded by the segment $[M_{m};M_{m+1}]$. We put $h=m-p+1$, and  we write $\Pi_{p,m}$ for the  canonical projection from $\R^{k+1}\times (\mathbb{S}^k)^{n+1}$ onto $\R^{k+1}\times (\mathbb{S}^k)^{h+1}$  defined as  $\Pi_{p,m}(x,z_1,\cdots,z_{n+1})=(x_{p-1},z_{p-1},z_p,\cdots,z_m)$ where $x_{p-1}$ are the Cartesian coordinates  of $M_{p-1}$.\\

The evolution of the extremity $M_{p-1}$ of this articulated sub-arm , controlled by the movement of  $[M_{m};M_{m+1}]$, is a  solution of the following  differential system (with  notations of Theorem \ref{system}):
\begin{eqnarray}\label{soussystemcont}
\begin{cases}\dot{x}_{p-1}^1= v_{p-1}\phi_{p-1}^1\cr
\dot{x}_{p-1}^2=v_{p-1}\phi_{p-1}^2\cr
\qquad\cdots\cr
\dot{x}_{p-1}^{k+1}=v_{p-1}\phi_{p-1}^{k+1}\cr
\dot{\theta}_{p-1}^1=v_{p}B_{p}^1\cr
\qquad\cdots\cr
\dot{\theta}_{p-1}^k=v_{p}B_{p}^k\cr
\qquad\cdots\\cr
\dot{\theta}_{i}^1=v_{i+1}B_{i+1}^1\cr
\qquad\cdots\cr
\dot{\theta}_{i}^k=v_{i+1}B_{i+1}^k\cr
\qquad\cdots\cr
\dot{\theta}_{m-1}^1=v_mB_{m}^1\cr
\qquad\cdots\cr
\dot{\theta}_{m-1}^k=v_mB_{m}^k\cr
\dot{\theta}_m^1=v_{\theta_m^1}=v_{m+1}B_{m+1}^1 \cr
\qquad\cdots\cr
\dot{\theta}_m^k=v_{\theta_m^k}=v_{m+1}B_{m+1}^k\cr
\end{cases}
\end{eqnarray}

It is a controlled system on  $\R^{k+1}\times(\mathbb{S}^k)^{h+1}$ $(h=m-p+1)$ :
$$\dot{\hat{q}}=u_0\hat{X}^0_h+\dis\sum_{i=1}^k u_i{X}^i_h$$
 with controls
$u_0=v_m$, and  $u_i=v_{\theta_m^i}$, for $i=1,\cdots,k$ \\
$\hat{q}=\Pi_{p,m}(q)$ \\
$\hat{X}_h^0=\dis\sum_{i=p}^m f^i_mZ_i+f^{p-1}_m\hat{Z}_{p-1}$ et $\hat{Z}_{p-1}=\dis\sum_{l=1}^k\phi_{p-1}^l\frac{\partial}{\partial x_{p-1}^l}$\\
We denote by $\hat{\D}_h$ the distribution generated by $\hat{X}_h^0$ and $X_m^1,\cdots,X_m^k$.\\
\end{Com}
 This comment will be used in a future paper about singular sets of special flags and their interpretations in terms of singularities kinematic evolution of an articulated arm.

\bigskip
\bigskip

\end{document}